\documentclass[11pt]{amsart}
\usepackage[all]{xy}
\usepackage{nicefrac}
\usepackage{amsmath}
\usepackage{amssymb}
\usepackage[margin=1.4in]{geometry}

\swapnumbers

\theoremstyle{plain}

\newtheorem{subprop}[subsubsection]{Proposition}
\newtheorem{sublemm}[subsubsection]{Lemma}
\newtheorem{fact}[subsubsection]{Fact}

\newtheorem{theo}[subsection]{Theorem}
\newtheorem{coro}[subsection]{Corollary}
\newtheorem{observ}[subsection]{Observation}
\newtheorem*{mainthm}{Theorem}

\def\t{\otimes}

\def\DD{\mathsf D}
\def\EE{\mathsf E}
\def\PP{\mathsf P}
\def\QQ{\mathsf Q}

\def\KK{\mathbb K}
\def\NN{\mathbb N}
\def\ZZ{\mathbb Z}

\def\we{\stackrel{\sim}{\rightarrow}}
\def\weinv{\stackrel{\sim}{\leftarrow}}

\DeclareMathOperator{\Der}{Der}
\DeclareMathOperator{\Hom}{Hom}
\DeclareMathOperator{\Ho}{Ho}

\title{Obstruction theory for algebras over an operad}
\author{Eric Hoffbeck}
\address{Fachbereich Mathematik, Universit\"at Hamburg, Bundesstra\ss e 55, 20146 Hamburg, Germany}
\email{Eric.Hoffbeck@math.uni-hamburg.de}

\begin{document}

\begin{abstract}
The goal of this paper is to set up an obstruction theory in the context of algebras over an operad and in the framework of differential graded modules over a field. 
Precisely, the problem we consider is the following: 
Suppose given two algebras $A$ and $B$ over an operad $\PP$ and an algebra morphism from $H_*A$ to $H_*B$. Can we realize this morphism as a morphism of $\PP$-algebras from $A$ to $B$ in the homotopy category? Also, if the realization exists, is it unique in the homotopy category?

We identify obstruction cocycles for this problem, and notice that they live in the first two groups of operadic $\Gamma$-cohomology.
\end{abstract}

\maketitle

\vspace{0.3cm}

In this paper we study a question of realizability of morphisms in a category of algebras over an operad.

In general, a realization problem takes the following form. We fix a category $\mathcal C$ equipped with a model structure (for instance: topological spaces, spectra, differential graded algebras over an operad). We have a homology (or homotopy) functor $H \colon \mathcal C \to \mathcal A$ with values in a purely algebraic category (for instance: graded modules, graded algebras). 
The usual questions are the existence of a realization of an object $a$ in $\mathcal A$ by a $c$ in $\mathcal C$ such that $H (c) = a$ and the existence of a realization of a morphism $f \colon H (c_1) \to H (c_2)$ by a morphism $\phi \colon c_1 \to c_2$ such that $H (\phi) = f$.





Generally, the obstructions to these existences can be interpreted as classes in some (co)homology theory.

The most classical example goes back to Steenrod for $\mathcal C$ the category of topological spaces and $H=H_{\text{sing}}^*$. A solution of this problem in the case of rational nilpotent CW-complexes has been given by Halperin and Stasheff in \cite{3HS}. 
They apply rational homotopy theory to reduce this topological realization problem to a realization problem in the category of differential graded commutative algebras.
The obstructions then live in some Harrison cohomology groups. The obstruction theory of Blanc, Dwyer and Goerss \cite{3BDG} for the realizability of $\Pi$-algebras by a space, the theories of Robinson \cite{3Rob} and of Goerss and Hopkins \cite{3GH} for the realizability of an algebra by an $E_\infty$-spectra are other fundamental examples of obstruction theory in homotopy theory.

\vspace{0.4cm}

We are here interested in the case $\mathcal C = {}^{}_\PP\text{dgMod}_\KK$, the category of algebras over a fixed operad $\PP$ in the framework of differential graded modules (for short dg-modules) over a field $\KK$. 
The functor $H$ is the homology  of the underlying dg-module of an algebra over $\PP$. This homology inherits a $H_*\PP$-algebra structure. The target category $\mathcal A$ consists of the graded $H_*\PP$-algebras. The realization problem has been studied by Livernet  \cite[Section 3]{3Liv} in the setting of $\NN$-graded dg-modules and when the ground field $\KK$ has characteristic $0$.
The obstruction classes live in some cohomology groups of a natural cohomology theory associated to $\PP$, generalizing Harrison cohomology for $\PP=\mathsf{Com}$.

In this paper, we obtain an obstruction theory for the realization of morphisms in the setting of $\ZZ$-graded dg-modules and when the ground ring $\KK$ is any field. We can identify a sequence of obstructions lying in some cohomology groups. Precisely, the $\Gamma$-cohomology of algebras over an operad (defined in \cite{3Hof}, generalizing Robinson's and Whitehouse's $\Gamma$-homology \cite{3RW}) appears in our construction and we get the following theorems:

\vspace{0.2cm}
\begin{mainthm}[Corollary \ref{threal}]
Let $\PP$ be a connected graded operad and let $\tilde\PP$ be an operadic cofibrant replacement of $\PP$. Let $A$ and $B$ be two algebras over $\tilde\PP$. 
Suppose given a $\PP$-algebra morphism $f : H_*A \to H_*B$ (where $H_*A$ and $H_*B$ have the structure induced in homology).

The obstruction cocycles to the realizability of the morphism $f$ lie in $H\Gamma_\PP^1(H_*A,H_*B)$. If $H\Gamma_\PP^1(H_*A,H_*B)$=0, then there automatically exists a morphism $\phi$ in the homotopy category of $\tilde \PP$-algebras such that $H_*\phi=f$. 
\end{mainthm}

\begin{mainthm}[Corollary \ref{thunicite}]
Let $\PP$ be a connected graded operad and let $\tilde\PP$ be an operadic cofibrant replacement of $\PP$. Let $A$ and $B$ be two algebras over $\tilde\PP$. 
Suppose given a $\PP$-algebra morphism $f : H_*A \to H_*B$ and two homotopy morphisms $\phi_1, \phi_2$ such that $H_*\phi_1=H_*\phi_2=f$.

The obstruction cocycles to the uniqueness of the realizations in the homotopy category lie in the group $H\Gamma_\PP^0(H_*A,H_*B)$. If $H\Gamma_\PP^0(H_*A,H_*B)=0$, then $\phi_1=\phi_2$ in the homotopy category of $\tilde \PP$-algebras.
\end{mainthm}

\vspace{0.3cm}

To obtain these theorems, the method is first to reduce our study to the case where the differentials of $A$ and $B$ are trivial. Then we use model category structures to make explicit cofibrant replacements of the algebras $A$ and $B$. The crucial point of the proof is a natural filtration of the cooperad $B(\PP \boxtimes \EE)$, which allows us to filter the cofibrant replacements. We construct step by step a map inducing the realization and identify the obstructions to this construction. 

An important thing to notice in our theorems is that only the structures of $\PP$-algebras on $H_*A$ and $H_*B$ appear. So we do not need to know the full structures on $A$ and $B$, but only a part of it.

\vspace{0.3cm}

There are some immediate corollaries to the previous theorems. First, one defines the set of homotopy automorphisms $\text{haut}_{\tilde\PP}(A) :=\{ \phi : A \we A \}$ for $A$ a cofibrant $\tilde\PP$-algebra.
We can consider its connected components for the following homotopy relation 

$$ \phi^0 \sim \phi^1 \ \Leftrightarrow \ \vcenter{
\xymatrix@M=3pt@H=18pt@W=8pt@R=18pt@C=38pt{
A \ar[dr]^{\sim \ \phi^0} \ar[d]_\sim &   \\
A \t \Delta^1 \ar@{.>}[r]^{\exists \phi^t}  & A. \\
A \ar[ur]_{\sim \ \phi^1} \ar[u]^\sim
}}$$
where $A \t \Delta^1$ denotes the cylinder object of $A$.

Consider the map $H_*(-): \pi_0(\text{haut}_{\tilde\PP}(A)) \to \text{aut}_\PP(H_*A)$.
Our obstruction theory implies the following results:
\begin{itemize}
\item If $H\Gamma_\PP^0(H_*A,H_*A)=0$ then $H_*(-)$ is injective.
\item If $H\Gamma_\PP^1(H_*A,H_*A)=0$ then $H_*(-)$ is surjective.
\end{itemize}

Moreover, for any $\PP$-algebra $H$, if $H\Gamma^1_\PP(H, H)=0$, the first theorem implies that all $\tilde\PP$-algebras $A$ such that $H_*(A) = H$ are connected by weak equivalences.

\vspace{1cm}
In Section 1, we recall some results about operads, cooperads and operadic $\Gamma$-homology. In Section 2, we identify the obstructions to the realization. In the last section, we study the obstructions to the uniqueness up to homotopy of the realizations.

\subsection*{Convention} We work in the differential graded setting. We take as ground category 
the category of differential $\ZZ$-graded modules (for short dg-modules) over a fixed field $\KK$. 

All operads $\PP$ will be assumed to be connected in the sense that
$\PP(0) = 0$ and $\PP(1) = \KK$.

\vspace{1.5cm}

\section{Recollections}

\subsection{Model structures}

We give references for the model structures of the categories which are used in this paper.
For general references on the subject, we refer the reader to the survey of Dwyer and Spalinksi \cite{3DS} and the books of Hirschhorn \cite{3Hirsch} and Hovey \cite{3Hovey}. For model structures in the operadic context, we refer to the articles of Hinich \cite{3Hinich}, of Berger and Moerdijk \cite{3BM1} and of Goerss and Hopkins \cite{3GH}, and the book of Fresse \cite{3BouquinBenoit}.

Just recall the following standard definitions:
\begin{enumerate}
 \item The category of dg-modules is equipped  with the model structure such that a morphism is a fibration (resp. a weak equivalence) if it is an epimorphism (resp. induces an isomorphism in homology). 
 \item The category of operads inherits a model structure where fibrations (resp. weak equivalences) are fibrations (resp. weak equivalences) of the underlying dg-modules. 
 \item The category of algebras over a cofibrant operad inherits a model structure where fibrations (resp. weak equivalences) are fibrations (resp. weak equivalences) of the underlying dg-modules. 
\end{enumerate}
In all cases, cofibrations are given by the LLP with respect to acyclic fibrations.

We usually call $\Sigma_*$-module a collection of dg-modules $\{M(r)\}_{r\in\NN}$ where each $M(r)$ is equipped with an action of the $r$-th symmetric group $\Sigma_r$. The category of $\Sigma_*$-modules also inherits a model structure such that fibrations (resp. weak equivalences) are fibrations (resp. weak equivalences) of the underlying dg-modules. Every operad has an underlying $\Sigma_*$-module and we say that an operad is $\Sigma_*$-cofibrant if the underlying $\Sigma_*$-module is cofibrant. The category of algebras over a $\Sigma_*$-cofibrant operad can also be equipped with a semi-model structure, but we will not need this refinement.

We will use a cofibrant replacement of operads given by the cobar-bar duality, which can be found in the paper of Getzler and Jones \cite{3GJ} in characteristic $0$, and the paper of Berger and Moerdijk \cite[Section 8.5]{3BM2} in our more general context.
We denote by $B$ the bar construction of an operad, introduced in \cite{3GK}, and by $B^c$ the cobar construction, introduced in \cite{3GJ}.
Recall that an element of the bar (or cobar) construction $B(\PP)$ can be seen as a tree labelled by elements of $\PP$. Thus the bar (and cobar) construction is equipped with a weight, given by the number of vertices of the tree representing an element. 
The operad $\EE$ denotes the Barratt-Eccles operad, whose definition is recalled later in Section \ref{actionE}, and $\boxtimes$ denotes the arity-wise tensor product of $\Sigma_*$-modules, i.e.  $(\PP \boxtimes \EE) (r) = \PP(r) \t \EE(r)$ for all $r \in \NN$.

\begin{fact}[{\cite[Theorem 8.5.4]{3BM2}}]
Let $\PP$ be an operad.

The operad $B^c(B(\PP \boxtimes \EE))$ is a cofibrant replacement of the operad $\PP$.
\end{fact}

If $\QQ$ is a cofibrant replacement of an operad $\PP$, working with algebras over $\QQ$ is equivalent to working with algebras over $B^c(B(\PP \boxtimes \EE))$. In this paper, we always pick this particular cofibrant replacement.

\subsection{Coalgebras over cooperads}
Let $\DD$ be a cooperad.
In the following series of propositions, we recall how the structure of $B^c(\DD)$-algebra on $A$ can be explicitely encoded in a quasi-cofree coalgebra $\DD(A)$.  We need precise formulas for our study. 

These results have been first given in the preprint of Getzler and Jones \cite{3GJ}. 
But we use them in the wider context of $\ZZ$-graded modules and over a field of any characteristic, and we refer to \cite{3Arolla2008} for the generalization in the latter setting.

Let $\DD$ be a cooperad and $A$ a dg-module.


We may represent an element $\gamma \in \DD(n)$ by a corolla with n inputs $\vcenter
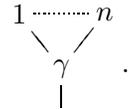
{
	\xymatrix@M=3pt@H=4pt@W=3pt@R=8pt@C=4pt{
	1 \ar@{-}[dr]\ar@{.}[rr] && n\ar@{-}[dl] \\
	&  \gamma\ar@{-}[d] &  \\
	&&}  } $.

We consider the total coproduct and the quadratic coproduct  of a cooperad structure, which send the element $\gamma$ to a composed element arranged on a tree.

The total coproduct denoted by $\nu$ maps an element $\gamma \in \DD$ to a sum of formal composites of elements represented by
$$\nu \left(\vcenter{
	\xymatrix@M=3pt@H=4pt@W=3pt@R=8pt@C=4pt{
	1 \ar@{-}[dr]\ar@{.}[rr] && n\ar@{-}[dl] \\
	&  \gamma\ar@{-}[d] &  \\
	&&}  }  \right) =
\sum_{\nu}\vcenter{
\xymatrix@M=3pt@H=4pt@W=3pt@R=8pt@C=4pt{
 i_{1,1}\ar@{.}[rr]\ar@{-}[dr] & & i_{1,s_1}\ar@{-}[dl] &	&i_{r,1}\ar@{.}[rr]\ar@{-}[dr] & & i_{r,s_r}\ar@{-}[dl] \\
&\gamma''_1\ar@{-}[drr]\ar@{.}[rrrr]&& & & \gamma''_r \ar@{-}[dll] & \\
	&&& \gamma'\ar@{-}[d] && \\
	&&&&}}$$
where $\gamma'$, $\gamma''_1, \ldots, \gamma''_r$ are elements of $\DD$ and the entries form a multi-shuffle of $\{1, \ldots, n \}$ 
(i.e. $ i_{1,1}<  i_{2,1} < \ldots <  i_{r,1}$ and  $i_{k,1} <  i_{k,2} < \ldots < i_{k,s_k}$ for all $1\leq k \leq r$).

To avoid too many indices, we will write such a sum in the following form:
$$\sum_{\nu}\vcenter{
\xymatrix@M=3pt@H=4pt@W=3pt@R=8pt@C=4pt{
 i_*\ar@{.}[rr]\ar@{-}[dr] & & i_*\ar@{-}[dl] &	&i_*\ar@{.}[rr]\ar@{-}[dr] & & i_*\ar@{-}[dl] \\
&\gamma''_*\ar@{-}[drr]\ar@{.}[rrrr]&& & & \gamma''_* \ar@{-}[dll] & \\
	&&& \gamma'\ar@{-}[d] && \\
	&&&&&&}}$$

The quadratic coproduct of an element $\gamma \in \DD$ is denoted by $\nu_2(\gamma)$ and represented by
$$\nu_2 \left(\vcenter{
	\xymatrix@M=3pt@H=4pt@W=3pt@R=8pt@C=4pt{
	1 \ar@{-}[dr]\ar@{.}[rr] && n\ar@{-}[dl] \\
	&  \gamma\ar@{-}[d] &  \\
	&&}  }  \right) =
\sum_{\nu_2}\vcenter{
\xymatrix@M=3pt@H=4pt@W=3pt@R=8pt@C=4pt{
 	&&j_1\ar@{.}[rr]\ar@{-}[dr] & & j_\ell\ar@{-}[dl] \\
&i_1\ar@{-}[dr]\ar@{.}[rr]& & \gamma''\ar@{-}[dl] \ar@{.}[r] & i_k\ar@{-}[dll] & \\
	&& \gamma'\ar@{-}[d] && \\
	&&&&}}$$
where $\gamma'$ and the $\gamma''$ are elements of $\DD$ and the $\{i_1, \ldots, i_k \} \coprod \{j_1, \ldots, j_\ell\}$ run over partitions of $\{1, \ldots, n \}$. 

Let $A$ be a dg-module, where the differential is denoted by $d_A$.
Recall that $\DD(A)$ is the cofree connected coalgebra given by $$\DD(A)= \bigoplus (\DD(r) \t A^{\t r})_{\Sigma_r}.$$
The element $\gamma(a_1, \ldots, a_r) \in \DD(A)$ is associated to the tensor $\gamma \t (a_1, \ldots, a_r)$. 
We represent an element in $\DD(A)$ by a corolla with inputs indexed by elements of $A$.

\begin{subprop}[{\cite[Proposition 2.14]{3GJ}, \cite[Proposition 4.1.3]{3Arolla2008}}]
For a cofree coalgebra $\DD(A)$, we have a bijective correspondance between $\DD$-coderivations $\partial : \DD(A) \to \DD(A)$ and homomorphisms $\alpha : \DD(A) \to A$. The homomorphism $\alpha$ associated to a coderivation $\partial$ is given by the compositive with the canonical projection. Conversely, the coderivation $\partial_\alpha$ associated to $\alpha$ is determined by 
$$\partial_\alpha \left(\vcenter{
	\xymatrix@M=3pt@H=4pt@W=3pt@R=8pt@C=4pt{
	a_1 \ar@{-}[dr]\ar@{.}[rr] && a_n\ar@{-}[dl] \\
	&  \gamma\ar@{-}[d] &  \\
	&&}  }  \right) =
\sum_i \pm \left(\vcenter{
	\xymatrix@M=3pt@H=4pt@W=3pt@R=8pt@C=4pt{
	a_1 \ar@{-}[dr]\ar@{.}[r] & \alpha(a_i)\ar@{-}[d]\ar@{.}[r]& a_n\ar@{-}[dl] \\
	&  \gamma\ar@{-}[d] &  \\
	&&}  }  \right)
+
\sum_{\nu_2} \pm \vcenter{
\xymatrix@M=3pt@H=4pt@W=3pt@R=8pt@C=4pt{
 	&a_*\ar@{.}[rr]\ar@{-}[dr] & \save+<-31pt,-11pt>*{\alpha \Biggl[} +<57pt,0pt>*{\Biggl]}\restore & a_*\ar@{-}[dl] \\
a_*\ar@{-}[drr]\ar@{.}[r]& & \gamma''\ar@{-}[d]  &\ar@{.}[r]& a_*\ar@{-}[dll] & \\
	&& \gamma'\ar@{-}[d] && \\
	&&&&}}$$
for every $\gamma(a_1, \ldots, a_n)$ in $\DD(A)$.
The first term corresponds to $\alpha$ applied to $a_i \in A \subset \DD(A)$. For the second term, we use the quadratic coproduct $\nu_2$ and then apply $\alpha$ on the upper corolla which represents an element in $\DD(A)$.
\end{subprop}

\begin{subprop}[{\cite[Proposition 4.1.4]{3Arolla2008}}]
Let $\alpha : \DD(A) \to A$ be a homomorphism of degree -1 such that $\alpha_{|A}=0$.

A $\DD$-coderivation of degree $-1$, $\partial_\alpha : \DD(A) \to \DD(A)$ so that $(\DD(A), \partial_\alpha)$ defines a differential graded quasi-cofree coalgebra if and only if the homomorphism $\alpha : \DD(A) \to A$ satisfies the relation
\begin{equation}\label{prop414}
\delta(\alpha) \left(\vcenter{
	\xymatrix@M=3pt@H=4pt@W=3pt@R=8pt@C=4pt{
	a_1 \ar@{-}[dr]\ar@{.}[rr] && a_n\ar@{-}[dl] \\
	&  \gamma\ar@{-}[d] &  \\
	&&}  }  \right) 
+
\sum_{\nu_2} \pm \alpha \left(\vcenter{
\xymatrix@M=3pt@H=4pt@W=3pt@R=8pt@C=4pt{
 	&a_*\ar@{.}[rr]\ar@{-}[dr] & \save+<-31pt,-11pt>*{\alpha \Biggl[} +<57pt,0pt>*{\Biggl]}\restore & a_*\ar@{-}[dl] \\
a_*\ar@{-}[drr]\ar@{.}[r]& & \gamma''\ar@{-}[d]  &\ar@{.}[r]& a_*\ar@{-}[dll] \\
	&& \gamma'\ar@{-}[d] && \\
	&&&}}\right)
=0
\end{equation}
for every element $\gamma(a_1, \ldots, a_n)$ in $\DD(A)$, where $\delta(\alpha)$ denotes $d_A \circ \alpha \pm \alpha \circ \partial_\alpha$.
\end{subprop}

\begin{subprop}[{\cite[proposition 2.15]{3GJ}, \cite[Proposition 4.1.5]{3Arolla2008}}]
A  $B^c(\DD)$-algebra structure on a dg-module $A$ is equivalent to a map $\alpha : \DD(A) \to A$ which satisfies Equation \eqref{prop414} and such that the restriction $\alpha_{|A}$ vanishes.
\end{subprop}

When we are given an operad morphism $B^c(\DD) \to \QQ$, we have a functor which, to any $\DD$-coalgebra $C$, associates a quasi-free $\QQ$-algebra $R_\QQ(C)=(\QQ(C), \partial)$ for some twisting differential $\partial$ (cf. \cite{3GJ} or \cite[Section 4.2.1]{3Arolla2008}).

We apply this construction to $\DD=B(\PP \boxtimes\EE)$, the morphism $id : B^c(\DD) \to B^c(\DD)=\tilde\PP$ and the coalgebra $C=(\DD(A),\partial_\alpha)$ associated to a $\tilde\PP$-algebra $A$ (the action is denoted by $\alpha$). We get the following result:

\begin{subprop}[{\cite[Theorem 2.19]{3GJ}, \cite[Theorem 4.2.4]{3Arolla2008}}]\label{rempcofib}
Let $A$ be an algebra over $\tilde\PP$ and let $\alpha$ denote the action. Let $D$ denote $B(\PP \boxtimes \EE)$.
The augmentation $\epsilon : R_{\tilde\PP}(\DD(A),\partial_\alpha)=(\tilde\PP(\DD(A),\partial_\alpha),\partial) \to A$ defines a weak equivalence and $(\tilde\PP(\DD(A),\partial_\alpha),\partial)$ forms a cofibrant replacement of $A$ in the category of $\tilde\PP$-algebras.
\end{subprop}


In this context, to study morphisms in the homotopy category of $\tilde\PP$-algebras, we just have to study morphisms of quasi-cofree $\DD$-coalgebras. The following two propositions show how to reduce our study to the corestrictions of such morphisms.

\begin{subprop}[{\cite[Observation 4.1.7]{3Arolla2008}}]\label{125}
The homomorphisms $\phi : \DD(A) \to \DD(B)$ of degree $0$ and commuting with coalgebra structures are in bijection with homomorphisms of dg-modules $f : \DD(A) \to B$. The homomorphism $f$ associated to $\phi$ is given by the composite of $\phi$ with the projection. Conversely, the homomorphism $\phi=\phi_f$ associated to $f$ is determined by the formula 
$$\phi_f\left(\vcenter{
	\xymatrix@M=3pt@H=4pt@W=3pt@R=8pt@C=4pt{
	a_1 \ar@{-}[dr]\ar@{.}[rr] && a_n\ar@{-}[dl] \\
	&  \gamma\ar@{-}[d] &  \\
	&&}  }  \right) 
=
\sum_{\nu}  \left(\vcenter{
\xymatrix@M=3pt@H=4pt@W=3pt@R=8pt@C=4pt{
 	a_*\ar@{.}[rr]\ar@{-}[dr] & \save+<-31pt,-11pt>*{f \Biggl[} +<57pt,0pt>*{\Biggl]}\restore & a_*\ar@{-}[dl] && a_*\ar@{.}[rr]\ar@{-}[dr] & \save+<-31pt,-11pt>*{f \Biggl[} +<57pt,0pt>*{\Biggl]}\restore & a_*\ar@{-}[dl]\\
 & \gamma''_*\ar@{-}[drr]  &\ar@{.}[rr]&&& \gamma''_*\ar@{-}[dll] & \\
	&&& \gamma'\ar@{-}[d] &&& \\
	&&&&}}\right)$$
for every element  $\gamma(a_1, \ldots, a_n)$ in $\DD(A)$.
We use the total coproduct  and we apply $f$ to all upper corrolas.
\end{subprop}

\begin{subprop}[{\cite[Proposition 4.1.8]{3Arolla2008}}]
The homomorphism of cofree coalgebras $\phi_f : \DD(A) \to \DD(B)$ associated to a homomorphism $f : \DD(A) \to B$ defines a morphism between quasi-cofree coalgebras $(\DD(A), \partial_\alpha) \to (\DD(B), \partial_\beta)$ if and only if we have the identity
$$\delta(f) \left(\vcenter{
	\xymatrix@M=3pt@H=4pt@W=3pt@R=8pt@C=4pt{
	a_1 \ar@{-}[dr]\ar@{.}[rr] && a_n\ar@{-}[dl] \\
	&  \gamma\ar@{-}[d] &  \\
	&&}  }  \right) 
-
\sum_{\nu_2} \pm f \left(\vcenter{
\xymatrix@M=3pt@H=4pt@W=3pt@R=8pt@C=4pt{
 	&a_*\ar@{.}[rr]\ar@{-}[dr] & \save+<-31pt,-11pt>*{\alpha \Biggl[} +<57pt,0pt>*{\Biggl]}\restore & a_*\ar@{-}[dl] \\
a_*\ar@{-}[drr]\ar@{.}[r]& & \gamma''\ar@{-}[d]  &\ar@{.}[r]& a_*\ar@{-}[dll] \\
	&& \gamma'\ar@{-}[d] && \\
	&&&}}\right)
$$ $$ +
\sum_{\nu} \beta \left(\vcenter{
\xymatrix@M=3pt@H=4pt@W=3pt@R=8pt@C=4pt{
 	a_*\ar@{.}[rr]\ar@{-}[dr] & \save+<-31pt,-11pt>*{f \Biggl[} +<57pt,0pt>*{\Biggl]}\restore & a_*\ar@{-}[dl] && a_*\ar@{.}[rr]\ar@{-}[dr] & \save+<-31pt,-11pt>*{f \Biggl[} +<57pt,0pt>*{\Biggl]}\restore & a_*\ar@{-}[dl]\\
 & \gamma''_*\ar@{-}[drr]  &\ar@{.}[rr]&&& \gamma''_*\ar@{-}[dll] & \\
	&&& \gamma'\ar@{-}[d] && \\
	&&&&}}\right)=0$$
for every element  $\gamma(a_1, \ldots, a_n)$ in $\DD(A)$.\end{subprop}

\subsection{The Barratt-Eccles operad and its action on cochains}\label{actionE}
Recall that an $E_\infty$-operad is a $\Sigma_*$-cofibrant replacement of the commutative operad.

The Barratt-Eccles operad $\EE$ is an example of $E_\infty$-operad, defined by the normalized chain complex $\EE = N_*(E\Sigma_\bullet)$, where $E\Sigma_n$ is the total space of the universal $\Sigma_n$-bundle in simplicial spaces. The chain complex $N_*(E\Sigma_n)$ is identified with the acyclic homogeneous bar construction of the symmetric group $\Sigma_n$, the module spanned in degree $t$ by the $(t+1)$-tuples of permutations $\underline w = (w_0, \ldots, w_t)$ together with the differential $\delta$ such that $\delta(\underline w) = \sum_i (-1)^i (w_0, \ldots, \widehat{w_i}, \ldots, w_t)$. We consider the left action of the symmetric group on this chain complex.

The composition product of $\EE$ is obtained using the composition product of permutations (which is just the insertion of a block). More precisely, for $\underline w = (w_0, \ldots, w_m) \in \EE(r)$ and $\underline w' = (w'_0, \ldots, w'_n) \in \EE(s)$, the composite $\underline w \circ_i \underline w' \in \EE(r+s-1)$ is defined by 
$$\underline w \circ_i \underline w' = \sum_{x_*, y_*} \pm (w_{x_0} \circ_i w'_{y_0}, \ldots, w_{x_{m+n}} \circ_i w'_{y_{m+n}})$$ where the sum ranges over the monotonic paths from $(0,0)$ to $(m,n)$ in $\NN \times \NN$.

The operad $\EE$ acts on $N^*(\Delta^1)$, according to the paper by Berger and Fresse \cite{3BF}. We denote this action by $\sigma$.
For our purposes, we simply recall the action of the component of degree $0$ of $\EE$. We have the equality of dg-modules $N^*(\Delta^1) = \KK. \underline 0^\# \oplus \KK. \underline 1^\# \oplus \KK. \underline{01}^\#$ where $\underline{0}^\#$, $\underline{1}^\#$ and $\underline{01}^\#$ denote the dual of the basis of non-degenerate simplices. The differential $\partial_N$ satisfies $\partial_N(\underline{01}^\#) = \underline{1}^\# - \underline{0}^\#$ and $\partial_N(\underline{0}^\#) = \partial_N(\underline{1}^\#)=0$. The $r$-th component in degree $0$ of $\EE$ is actually $\Sigma_r$, and the identity of $\Sigma_r$ acts on $N^*(\Delta^1)$ as follows:
\begin{itemize}
 \item $id . (\underline 0^\#,\ldots, \underline 0^\#, \underline {01}^\#, \underline 1^\#, \ldots, \underline 1^\#) = \underline {01}^\#$
 \item $id . (\underline 0^\#,\ldots, \underline 0^\#) = \underline 0^\#$
 \item $id . (\underline 1^\#,\ldots, \underline 1^\#) = \underline 1^\#$
 \item $id . (u_1, \ldots, u_r)=0$ otherwise.
\end{itemize}
The equivariance gives the action of the other permutations of $\Sigma_r$. We will not need the formula for the action of $\EE$ in higher degrees.

\subsection{The cylinder object of an algebra over an operad}\label{cylindre}

Let $\QQ$ be any cofibrant operad, for instance $\QQ= B^c(B(\PP\boxtimes\EE))$.
Let $B$ be a $\QQ$-algebra, with the structure given by $\beta$. We recall in this section the results we need from  \cite[Section 3.1]{3BF}.

The cylinder object of $B$ in the category of $\QQ$-algebras is $B \t N^*(\Delta^1)$. 

It is naturally endowed with the action $\beta \t \sigma$ of $\QQ \boxtimes \EE$:
$$(q \t \pi) (b_1 \t u_1, \ldots, b_r \t u_r) = q(b_1, \ldots, b_r) \t \pi (u_1, \ldots, u_r)$$
for $q \in \QQ, \pi \in \EE, (b_1, \ldots, b_r) \in B^r, (u_1, \ldots, u_r) \in N^*(\Delta^1)^r$. Fixing an operadic section $\rho : \QQ \to \QQ \boxtimes \EE$ of the augmentation $\QQ \boxtimes \EE \to \QQ$, we can see $B \t N^*(\Delta^1)$ as a $\QQ$-algebra. In Section \ref{section}, we will fix an explicit map $\rho$.

\subsection{Operadic $\Gamma$-cohomology}\label{rappelgamma}

In \cite{3Hof}, we have defined a generalization of Robinson's and Whitehouse's $\Gamma$-homology. We recall the definition here in the context of this paper.

Let $A$ and $B$ be $\PP$-algebras and $f: A \to B$ a morphism of $\PP$-algebras. 
The $\Gamma$-(co)homology of $A$ with coefficients in $B$ is defined by  $H_*(\Der_{\tilde \PP} (\tilde A, B))$ where $\tilde\PP$ is a $\Sigma_*$-cofibrant replacement of $\PP$ and $\tilde A$ a cofibrant replacement of $A$ as $\tilde\PP$-algebras. In this definition, the derivations are the $\tilde \PP$-derivations relatively to the morphism $\psi \circ f$, where $\psi$ denotes the morphism $\tilde A \we A$. An easy way to understand this definition is the following: the $\Gamma$-cohomology of a $\PP$-algebra $A$ is the usual Andr\'e-Quillen cohomology of $A$ seen as an algebra over a $\Sigma_*$-cofibrant replacement of $\PP$.


\vspace{1.5cm}

\section{Realizations of morphisms}
Suppose given 
\begin{itemize}
        \item an operad $\PP$ with the canonical operadic cofibrant replacement $\tilde\PP=B^c(B(\PP\boxtimes\EE))$,

	\item two algebras, $A$ and $B$, over $\tilde\PP$, 
	\item a $\PP$-algebra morphism $f_0$ :  $H_*A \to H_*B$ (where $H_*A$ and $H_*B$ have the structure induced in homology).
\end{itemize}
We want to understand the obstruction to the existence of a morphism $\phi: A \to B$ in the homotopy category of $\tilde\PP$-algebras such that $H_*\phi=f_0$.

\subsection{Outline of the study}\label{layout1}
We will proceed in the following way:

We first show in Section~\ref{restrict} that we can restrict our study to the case where the differentials of $A$ and $B$ are trivial, and we give some results concerning the structures induced in homology.
We consider the cooperad $\DD = B(\PP \boxtimes \EE)$, and the explicit cofibrant replacements of $A$ and $B$ from Proposition~\ref{rempcofib}. 
In Section~\ref{construct1}, we want to construct a $\DD$-coalgebra map $\phi_f : (\DD(A), \partial_{\alpha}) \to (\DD(B), \partial_{\beta})$ extending $f_0$ (it will lead to the expected morphism in the homotopy category). We introduce a filtration on $\DD(A)$, to proceed by induction. 
We notice that the obstruction to the construction of $\phi_f$ lies in a certain cohomology group which can be identified with the first group of $\Gamma$-cohomology of $H_*A$ with coefficients in $H_*B$.
If $\phi _{f}$ can be constructed, then (as the construction $R_{\tilde\PP}$ is functorial, see Section~\ref{rempcofib}) we obtain $\tilde \PP \phi _{f}:=R_{\tilde\PP}(\phi_f)$ which fits a diagram
$$\xymatrix@M=3pt@H=4pt@W=3pt@R=18pt@C=34pt{
(\tilde\PP(\DD(A), \partial_\alpha),\partial) \ar[r]^{\tilde \PP \phi _{f}}\ar[d]^\sim  & (\tilde \PP(\DD(B), \partial_\alpha),\partial) \ar[d]^\sim \\
A & B
}.$$
and thus we obtain a morphism from $A$ to $B$ in the homotopy category of $\tilde\PP$-algebras.

\subsection{Restriction of the hypotheses}\label{restrict}
We show here that we can reduce our study to the case where the differentials of $A$ and $B$ are trivial. 

First, recall the following result concerning the transfer of structures:
\begin{fact}\label{transpstruct}
Let $f : A \we B$ be a weak equivalence of dg-modules. Suppose that $B$ has an action of a cofibrant operad $\QQ$.

Then $A$ inherits the structure of a $\QQ$-algebra such that 
\begin{enumerate}
 \item $A \weinv \cdot \we B$ where the morphisms are weak equivalences of $\QQ$-algebras,
 \item $H_* (A \weinv \cdot \we B)=H_*f$.
\end{enumerate}
\end{fact}
This result in the $A_\infty$ context was already in Kadeishvili's work \cite{3Kad}. 
In our context, we refer to the result stated by Fresse \cite[Theorem A]{3Fgeorg}. The second assertion is not made explicit in the theorem but follows immediately from the proof.

Let $H=H_*A$ be the homology of a $\QQ$-algebra $A$. The graded module $H$ can be seen as a dg-module with a trivial differential, weakly equivalent to $A$ as dg-modules. 
We fix a splitting $A_* = Z_*A \oplus B'_{*-1}A$, where $Z_*A$ denote the cycles of $A$ (and where $B'_{*-1}A$ is isomorphic to the boundaries $B_{*-1}A$). This yields a map $A \to Z_*A$, which induces a map $A \to H$ by composition with the projection $Z_*A \to H$.
As we are working over a field, we can fix a section of dg-modules $s_A : H_*A \to Z_*A$ of the projection $Z_*A \twoheadrightarrow H_*A$, and thus a map $H \we A$.

The fact \ref{transpstruct} implies that $H$ inherits a structure of a $\QQ$-algebra such that $H \weinv \cdot \we A$, where the morphisms are weak equivalences of $\QQ$-algebras. This action of $\QQ$ on $H$ induces in homology an action of $H_*\QQ$ on $H=H_*H$.

On the other hand, as $H$ is the homology of the $\QQ$-algebra $A$, it inherits the structure of an algebra over $H_*\QQ$.

\begin{sublemm}
The two actions of $H_*\QQ$ on $H$ defined above coincide.
\end{sublemm}

\begin{proof}
The zig-zag of $\QQ$-algebras $H \weinv \cdot \we A$ induces in homology the zig-zag of $H_*\QQ$-algebras  $H \stackrel{\simeq}{\leftarrow} H_*(\cdot) \stackrel{\simeq}{\rightarrow} H_*A$. By the second point of the Fact \ref{transpstruct}, $H$ (with the first action) and $H_*A$ (with the second action) are equal as $H_*\QQ$-algebras.
\end{proof}


Let $B$ be another $\QQ$-algebra and $K=H_*B$ its homology. Let $\tilde H$ and $\tilde K$ be cofibrant replacements of $H$ and $K$ in the category of $\QQ$-algebras.
We use the following identities
$$\Hom_{\Ho\QQ-alg}(A,B)=\Hom_{\Ho\QQ-alg}(H,K)=[\tilde H,\tilde K]_{\QQ-alg}$$
where the notation $[-,-]$ refers to the homotopy classes, to restrict our study to the case of trivial differentials. 

\vspace{1cm}

Let $\alpha$ denote the action of the operad $\QQ$ on the dg-module $A$. We now make explicit the action $\alpha_1$ of $H_*\QQ$ on $H_*A$.

Let $Z_*\QQ$ denote the cycles of $\QQ$. As before, we can consider a section of the homology $s_\QQ : H_*\QQ \to Z_*\QQ$.

\begin{observ}
The action $\alpha_1$ can be determined by the commutativity of the following diagram:
$$\xymatrix@M=3pt@H=18pt@W=18pt@R=18pt@C=58pt{
H_*\QQ(r) \t H_*A^{\t r}\ar[r]^{\alpha_1} \ar[d]^{s_\QQ \t (s_A)^{\t r}} & H_*A \\
Z_*\QQ(r) \t Z_*A^{\t r}\ar@{.>}[r]^{\alpha} \ar[d]  & Z_*A \ar[u]\ar[d] \\
\QQ(r) \t A^{\t r}\ar[r]^{\alpha}  & A.  \\
},$$
where the dotted map is the restriction. The image of this restriction is included in the cycles of $A$.
\end{observ}

We now consider the case where $A$ is an algebra over $\QQ=\tilde\PP:=B^c(B(\PP\boxtimes\EE))$, where $\PP$ is a graded operad.
We use the particular section $\PP \hookrightarrow B^c(B(\PP\boxtimes\EE))$ given by the composite of the inclusion $\PP \to \PP \boxtimes \EE$ (sending $p \in \PP(r)$ to $p \t id_{\Sigma_r}$), with the obvious inclusions $\PP \boxtimes \EE$ to $B(\PP \boxtimes \EE)$ and $B(\PP \boxtimes \EE) \to B^c(B(\PP\boxtimes\EE))$. The above paragraphs give an action of $\PP$ on $H_*A$. If $\delta_A=0$, then we identify $A$ and $H_*A$, and thus we obtain the action $\alpha_1$ of $\PP$ on $A$.

\subsection{Construction of the morphism of coalgebras}\label{construct1}

We can now study our problem. We are given
\begin{itemize}
        \item a differential graded operad $\PP$ such that $\delta_\PP$=0,
	\item two algebras, $A$ and $B$, over $\tilde\PP=B^c(B(\PP\boxtimes\EE))$, with actions denoted by $\alpha$ and $\beta$, with trivial differentials,
	\item a $\PP$-algebra morphism $f_0$ :  $(H_*A, \alpha_1) \to (H_*B,\beta_1)$.
\end{itemize}
In this section, we do not distinguish between $A$ (resp. $B$) and $H_*A$  (resp. $H_*B$) as they are equal as dg-modules. We specify the structure ($\alpha$ or $\alpha_1$, $\beta$ or $\beta_1$) when we consider them as algebras over $\tilde\PP$ or $\PP$.

We want to define a morphism $\phi_{f}$ of $\DD$-coalgebras from $(\DD(A), \partial_{\alpha})$ to $(\DD(B), \partial_{\beta})$ such that the first component for a certain graduation is $f_0$. Recall from Section \ref{layout1} that such a morphism $\phi_f$ will induce  a morphism from $A$ to $B$ in the homotopy category. The morphism $\phi_f :(\DD(A), \partial_{\alpha}) \to (\DD(B), \partial_{\beta})$ will be the morphism induced by $f:\DD(A) \to B$, as defined in Proposition \ref{125}.

We use the graduation of $\DD=B(\PP \boxtimes \EE)$ given by the sum of the bar weight and the degree in $\EE$. This graduation of $\DD$ induces a splitting $\DD(A) =\bigoplus_d \DD_{[d]}(A)$ (we do not take into account any degree of $A$ or weight in $A$). The quadratic coproduct $\nu_2$ on $\DD$ sends $\gamma \in \DD_{[d+1]}$ to composites 
$$\vcenter
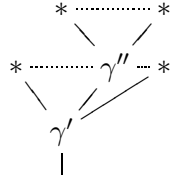
{
\xymatrix@M=3pt@H=4pt@W=3pt@R=8pt@C=4pt{
 	&&\ast\ar@{.}[rr]\ar@{-}[dr] & & \ast\ar@{-}[dl] \\
&\ast\ar@{-}[dr]\ar@{.}[rr]& & \gamma''\ar@{-}[dl] \ar@{.}[r] & \ast\ar@{-}[dll] & \\
	&& \gamma'\ar@{-}[d] && \\
	&&&&}}$$
such that $\gamma' \in \DD_{[p]}$, $\gamma'' \in \DD_{[q]}$ and $p+q=d+1$.

We want to construct the map $f$ by induction on the degree.
We notice that in degree zero, $\DD_{[0]}(A)$ is reduced to $A$ and thus we define ${f}_{[0]} = f_0$ (remember we want $\phi_f$ to realize $f_0$).

The morphism $\phi_f$ must fit the following commutative diagram:

$$
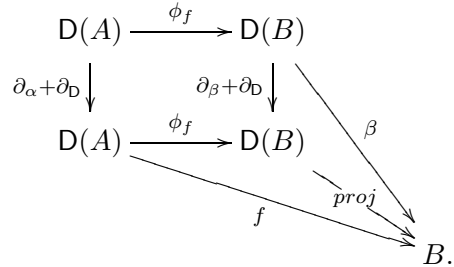
\xymatrix@M=3pt@H=18pt@W=8pt@R=18pt@C=38pt{
\DD(A)\ar[r]^{\phi_{f}} \ar[d]_{\partial_{\alpha}+\partial_\DD} & \DD(B)\ar[d]_{\partial_{\beta}+\partial_\DD}\ar[ddr]^{\beta} & \\
\DD(A)\ar[r]^{\phi_{f}}\ar[drr]_{f}  & \DD(B)\ar[dr]|{proj} & \\
& & B.
}$$
The triangle on the right obviously commutes. The commutativity of the triangle on the left defines $f$, the restriction of $\phi_f$ at the target. The commutativity of the exterior diagram is equivalent to the commutativity of the inner square.

The commutativity of this diagram is equivalent to the equation:
\begin{equation}\label{eq1}
f \circ (\partial_\DD + \partial_{\alpha}) = \beta \circ \phi_{f}. 
\end{equation}

We now suppose that $f$ is defined for degrees smaller than $d$ and we consider an element $\gamma (a_1, \ldots, a_n)$ where $\gamma$ lies in $\DD_{[d+1]}$. For this element, Equation \eqref{eq1} is equivalent to 
$$f \left(\vcenter{
	\xymatrix@M=3pt@H=4pt@W=3pt@R=8pt@C=4pt{
	a_1 \ar@{-}[dr]\ar@{.}[rr] && a_n\ar@{-}[dl] \\
	&  \partial_{\DD}\gamma\ar@{-}[d] &  \\
	&&}  }  \right) 
+
\sum_{\nu_2} \sum_{k=1}^{d} f \left(\vcenter{
\xymatrix@M=3pt@H=4pt@W=3pt@R=8pt@C=4pt{
 	&a_*\ar@{.}[rr]\ar@{-}[dr] & \save+<-33pt,-11pt>*{\alpha \Biggl[} +<61pt,0pt>*{\Biggl]}\restore & a_*\ar@{-}[dl] \\
a_*\ar@{-}[drr]\ar@{.}[r]& & \gamma''_{[k]}\ar@{-}[d]  &\ar@{.}[r]& a_*\ar@{-}[dll] \\
	&& \gamma'\ar@{-}[d] && \\
	&&&}}\right)
$$ $$ =
\sum_{\nu} \sum_{k=1}^{d+1} \beta \left(\vcenter{
\xymatrix@M=3pt@H=4pt@W=3pt@R=8pt@C=4pt{
 	a_*\ar@{.}[rr]\ar@{-}[dr] & \save+<-31pt,-11pt>*{f \Biggl[} +<57pt,0pt>*{\Biggl]}\restore & a_*\ar@{-}[dl] && a_*\ar@{.}[rr]\ar@{-}[dr] & \save+<-31pt,-11pt>*{f \Biggl[} +<57pt,0pt>*{\Biggl]}\restore & a_*\ar@{-}[dl]\\
 & \gamma''_*\ar@{-}[drr]  &\ar@{.}[rr]&&& \gamma''_*\ar@{-}[dll] & \\
	&&& \gamma'_{[k]}\ar@{-}[d] && \\
	&&&&}}\right)$$
where $\gamma'_{[k]}$ and $\gamma''_{[k]}$ denote elements in $\DD_{[k]}$.

Specifying the degrees of $f$ and taking the terms for $k=1$ out of the sums, we get:
$$f_{[d]} \left(\vcenter{
	\xymatrix@M=3pt@H=4pt@W=3pt@R=8pt@C=4pt{
	a_1 \ar@{-}[dr]\ar@{.}[rr] && a_n\ar@{-}[dl] \\
	&  \partial_{\DD}\gamma\ar@{-}[d] &  \\
	&&}  }  \right) 
+
\sum_{\nu_2} f_{[d]} \left(\vcenter{
\xymatrix@M=3pt@H=4pt@W=3pt@R=8pt@C=4pt{
 	&a_*\ar@{.}[rr]\ar@{-}[dr] & \save+<-35pt,-11pt>*{\alpha_0 \Biggl[} +<63pt,0pt>*{\Biggl]}\restore & a_*\ar@{-}[dl] \\
a_*\ar@{-}[drr]\ar@{.}[r]& & \gamma''_{[1]}\ar@{-}[d]  &\ar@{.}[r]& a_*\ar@{-}[dll] \\
	&& \gamma'\ar@{-}[d] && \\
	&&&}}\right)
$$ $$
+
\sum_{\nu_2} \sum_{k=2}^{d} f_{[d+1-k]} \left(\vcenter{
\xymatrix@M=3pt@H=4pt@W=3pt@R=8pt@C=4pt{
 	&a_*\ar@{.}[rr]\ar@{-}[dr] & \save+<-33pt,-11pt>*{\alpha \Biggl[} +<61pt,0pt>*{\Biggl]}\restore & a_*\ar@{-}[dl] \\
a_*\ar@{-}[drr]\ar@{.}[r]& & \gamma''_{[k]}\ar@{-}[d]  &\ar@{.}[r]& a_*\ar@{-}[dll] \\
	&& \gamma'\ar@{-}[d] && \\
	&&&}}\right)
$$ $$ =
\beta_1\left(\vcenter{
\xymatrix@M=3pt@H=4pt@W=3pt@R=8pt@C=4pt{
 	&a_*\ar@{.}[rr]\ar@{-}[dr] & \save+<-37pt,-11pt>*{f_{[d]} \Biggl[} +<65pt,0pt>*{\Biggl]}\restore & a_*\ar@{-}[dl] \\
f_0 a_*\ar@{-}[drr]\ar@{.}[r]& & \gamma''_{[d]}\ar@{-}[d]  &\ar@{.}[r]& f_0 a_*\ar@{-}[dll] \\
	&& \gamma'_{[1]}\ar@{-}[d] && \\
	&&&}}\right)
+
\sum_{\nu} \sum_{k=2}^{d+1} \beta \left(\vcenter{
\xymatrix@M=3pt@H=4pt@W=3pt@R=8pt@C=4pt{
 	a_*\ar@{.}[rr]\ar@{-}[dr] & \save+<-31pt,-11pt>*{f \Biggl[} +<57pt,0pt>*{\Biggl]}\restore & a_*\ar@{-}[dl] && a_*\ar@{.}[rr]\ar@{-}[dr] & \save+<-31pt,-11pt>*{f \Biggl[} +<57pt,0pt>*{\Biggl]}\restore & a_*\ar@{-}[dl]\\
 & \gamma''_*\ar@{-}[drr]  &\ar@{.}[rr]&&& \gamma''_*\ar@{-}[dll] & \\
	&&& \gamma'_{[k]}\ar@{-}[d] && \\
	&&&&}}\right).$$

The last sum of the left hand side and the last sum of the right hand side involve $f$ in degrees smaller than $d$, while the three other terms involve $f$ only in degree exactly $d$. The second and fourth terms involve respectively $\alpha_0$ and $\beta_0$, as only the restricted structure matters for elements in degree $0$. 

Thus we write the above equation in the following form: 
$$f_{[d]} \left(\vcenter{
	\xymatrix@M=3pt@H=4pt@W=3pt@R=8pt@C=4pt{
	a_1 \ar@{-}[dr]\ar@{.}[rr] && a_n\ar@{-}[dl] \\
	&  \partial_{\DD}\gamma\ar@{-}[d] &  \\
	&&}  }  \right) 
+
\sum_{\nu_2} f_{[d]} \left(\vcenter{
\xymatrix@M=3pt@H=4pt@W=3pt@R=8pt@C=4pt{
 	&a_*\ar@{.}[rr]\ar@{-}[dr] & \save+<-35pt,-11pt>*{\alpha_0 \Biggl[} +<63pt,0pt>*{\Biggl]}\restore & a_*\ar@{-}[dl] \\
a_*\ar@{-}[drr]\ar@{.}[r]& & \gamma''_{[1]}\ar@{-}[d]  &\ar@{.}[r]& a_*\ar@{-}[dll] \\
	&& \gamma'\ar@{-}[d] && \\
	&&&}}\right)
$$ $$ -
\sum_{\nu} \beta_1\left(\vcenter{
\xymatrix@M=3pt@H=4pt@W=3pt@R=8pt@C=4pt{
 	&a_*\ar@{.}[rr]\ar@{-}[dr] & \save+<-37pt,-11pt>*{f_{[d]} \Biggl[} +<65pt,0pt>*{\Biggl]}\restore & a_*\ar@{-}[dl] \\
f_0 a_*\ar@{-}[drr]\ar@{.}[r]& & \gamma''_{[d]}\ar@{-}[d]  &\ar@{.}[r]& f_0 a_*\ar@{-}[dll] \\
	&& \gamma'_{[1]}\ar@{-}[d] && \\
	&&&}}\right)
$$ $$
=-
\sum_{\nu_2} \sum_{k=2}^{d} f_{[d+1-k]} \left(\vcenter{
\xymatrix@M=3pt@H=4pt@W=3pt@R=8pt@C=4pt{
 	&a_*\ar@{.}[rr]\ar@{-}[dr] & \save+<-33pt,-11pt>*{\alpha \Biggl[} +<61pt,0pt>*{\Biggl]}\restore & a_*\ar@{-}[dl] \\
a_*\ar@{-}[drr]\ar@{.}[r]& & \gamma''_{[k]}\ar@{-}[d]  &\ar@{.}[r]& a_*\ar@{-}[dll] \\
	&& \gamma'\ar@{-}[d] && \\
	&&&}}\right)
$$ $$+
\sum_{\nu} \sum_{k=2}^{d+1} \beta \left(\vcenter{
\xymatrix@M=3pt@H=4pt@W=3pt@R=8pt@C=4pt{
 	a_*\ar@{.}[rr]\ar@{-}[dr] & \save+<-31pt,-11pt>*{f \Biggl[} +<57pt,0pt>*{\Biggl]}\restore & a_*\ar@{-}[dl] && a_*\ar@{.}[rr]\ar@{-}[dr] & \save+<-31pt,-11pt>*{f \Biggl[} +<57pt,0pt>*{\Biggl]}\restore & a_*\ar@{-}[dl]\\
 & \gamma''_*\ar@{-}[drr]  &\ar@{.}[rr]&&& \gamma''_*\ar@{-}[dll] & \\
	&&& \gamma'_{[k]}\ar@{-}[d] && \\
	&&&&}}\right)$$
with $f$ in degree $d$ grouped in the left hand side and $f$ in degrees smaller than $d$ grouped in the right hand side.

According to our induction hypothesis, the right hand side is known. The left hand side can be identified with $\partial(f_{[d]}(\gamma))$ where $\partial$ is the differential in $\Der_{\tilde\PP}(\tilde\PP(\DD (A), \partial_{\alpha_1}), (B,\beta_1))$
and $\gamma \in \DD$ is identified with $1_{\tilde\PP} \circ \gamma \in \tilde\PP\DD$. Note that these derivations take into account only the restricted structures $\alpha_1$ and $\beta_1$, and not the full structures $\alpha$ and $\beta$.

We have proved
\begin{theo}
If the cohomology group $H^1\Der_{\tilde\PP}(\tilde\PP(\DD (A), \partial_{\alpha_1}), (B,\beta_1))$ is equal to $0$, we can construct a map $f_{[d]}$ (i.e. continue our induction), and hence a map $\phi_f$ answering the initial problem. 
\end{theo}

We now relate this cohomology group with one group of $\Gamma$-homology:

\begin{coro}\label{threal}
The obstruction to the realizability of morphisms lies in $H\Gamma_\PP^1(H_*A,H_*B)$. 
\end{coro}

\begin{proof}
The $\tilde\PP$-algebra $\tilde\PP(\DD (A),\partial_{\alpha_1})$ is nothing but a cofibrant replacement of $(A,\alpha_1)$ (cf. Proposition \ref{rempcofib}), so the cohomology $H^*\Der_{\tilde\PP}(\tilde\PP(\DD (A),\partial_{\alpha_1}), (B,\beta_0))$ is the $\Gamma$-cohomology of the $\tilde\PP$-algebra $A$ with coefficients in $B$, for the actions $\alpha_1$ and $\beta_1$. This cohomology is actually $H\Gamma_\PP^*(H_*A,H_*B)$, cf. Section \ref{rappelgamma}.
\end{proof}

\subsection{Remarks.}\label{remarks}
\begin{itemize}
 \item The $d$-th obstruction lies in $H^1\Der_{\tilde\PP}(\tilde\PP(\DD_{[d]} (A), \partial_{\alpha_1}), (B,\beta_1))$. 
Thus the total obstruction lies in $\bigoplus_d H^1\Der_{\tilde\PP}(\tilde\PP(\DD_{[d]} (A), \partial_{\alpha_1}), (B,\beta_1))$. 

\noindent Note that $\bigoplus_d H^1\Der_{\tilde\PP}(\tilde\PP(\DD_{[d]} (A), \partial_{\alpha_1}), (B,\beta_1))$ is included in $H\Gamma_\PP^{1}(H_*A,H_*B)$ but has no reason to be equal.
 \item It is possible to work over a ring $\KK$ instead of a field, but some additional assumptions are then necessary. We need to assume that relevant dg-modules over $\KK$ are projective and that we have sections of the maps: $H_*A \to A$ and $H_*B \to B$. 
 
\end{itemize}

\vspace{1.5cm}

\section{Realization of homotopies}

In this section, we consider the problem of uniqueness of realizations in the homotopy category.
We are given 
\begin{itemize}
        \item a graded operad $\PP$ with the canonical operadic cofibrant replacement $\tilde\PP=B^c(B(\PP\boxtimes\EE))$;
	\item two algebras over $\tilde\PP$, $(A, \alpha)$ and $(B, \beta)$, 
        \item two morphisms $f^0 ,f^1 : \DD(A) \to B$ realizing the same $\PP$-algebra morphism $\psi : H_*A \to H_*B$.
\end{itemize}
The morphisms $f^0$ and $f^1$ induce morphisms  $\tilde\PP \phi_{f^0}$ and  $\tilde\PP \phi_{f^1}$ from $\tilde\PP\DD (A)$ to $\tilde\PP\DD (B)$, and thus two morphisms of $\tilde\PP$-algebras from $A$ to $B$ in the homotopy category.
The question we want to study in this section is: what is the obstruction to the equality of these morphisms in the homotopy category?
We show that the obstruction lies in a group of $\Gamma$-cohomology.

\vspace{1cm}

\subsection{Outline of the study}

We restrict our study to the case where the differentials of $A$ and $B$ are trivial. We consider the cooperad $\DD$ defined by $B(\PP \boxtimes \EE)$. We also consider the cylinder object $B \t N^*(\Delta^1)$ of $B$ in the category of $\tilde\PP$-algebras, whose action is denoted $(\beta \t \sigma) \circ \rho$, cf. Section \ref{cylindre}. For this matter, we define an explicit section $\rho :\tilde\PP \to \tilde\PP \boxtimes \EE$ in Section \ref{section}.

In Section \ref{construct2}, we want to construct a $\DD$-coalgebra map 
$\phi_f : (\DD(A), \partial_\alpha) \to (\DD(B \t N^*(\Delta^1)), \partial_{(\beta \t \sigma)\circ \rho})$
giving a homotopy between $\phi_{f^0}$ and $\phi_{f^1}$. Its restriction $f$ must fit into the following commutative diagram:

$$\xymatrix@M=3pt@H=4pt@W=3pt@R=18pt@C=34pt{
 & B\ar[d]^{i_0} \\
\DD(A)\ar[ur]^{f^0}\ar[r]^{f}\ar[dr]_{f^1}     &     B \t N^*(\Delta^1) \\
 & B\ar[u]_{i_1} \, \,  .
} $$

As in the previous section, we will construct $\phi_f$ by induction, and see the obstruction to the construction. Such a map $\phi_f$ induces a homotopy between  the morphisms $\tilde\PP \phi_{f^0}$ and $\tilde\PP \phi_{f^1}$ and thus their equality in the homotopy category.
Our study is very similar to the previous one, except we have to consider the cylinder object $B\t N^*(\Delta^1)$ instead of $B$ itself. 

\vspace{1.5cm}

\subsection{Explicitation of a section}\label{section}

We define in this section an explicit operadic section $\rho : \tilde\PP \to \tilde\PP \boxtimes \EE$.

Recall from \cite{3BM2} that the cobar-bar construction $B^c(B(-))$ can be identified with the cubical W-construction $W_\square(-)$.
Markl and Shnider \cite{3MS} have constructed a diagonal on the $W$-construction: a map $W_{\square}(\QQ) \xrightarrow{\Delta_\QQ} W_\square(\QQ) \boxtimes W_\square(\QQ)$ for any operad $\QQ$.

On the other hand, we easily observe that for any operads $\PP$ and $\QQ$, we can identify $B^c(B(\PP\boxtimes \QQ))$ and $B^c(B(\PP)) \boxtimes B^c(B(\QQ))$. 

Combining these two facts, we can consider the composite :
\begin{eqnarray*}
 B^c(B(\PP\boxtimes \EE)) & = & B^c(B(\PP)) \boxtimes B^c(B(\EE)) \\
 & \stackrel{id \boxtimes \Delta_\EE}{\to} & B^c(B(\PP)) \boxtimes B^c(B(\EE))  \boxtimes B^c(B(\EE)) \\
 & = & B^c(B(\PP\boxtimes \EE)) \boxtimes B^c(B(\EE)) \\
 & \stackrel{id \boxtimes aug}{\to} & B^c(B(\PP\boxtimes \EE)) \boxtimes \EE
\end{eqnarray*}
where $aug$ denotes the augmentation $B^c(B(\EE)) \to \EE$.

We denote this composite by $\rho : \tilde\PP \to \tilde\PP \boxtimes \EE$.

\subsection{Construction of the morphism of coalgebras}\label{construct2}

Suppose $A$ and $B$ are algebras over $\tilde\PP$. The same argument as in Section \ref{restrict} allows us to suppose their differentials are trivial. We use the same graduation as in Section \ref{construct1}.

The morphism $\phi_f$ must fit the following commutative diagram:

$$\xymatrix@M=3pt@H=18pt@W=38pt@R=28pt@C=48pt{
\DD(A)\ar[r]^{\phi_{f}} \ar[d]_{\partial_{\alpha}+\partial_\DD} & 
         \DD(B\t N^*(\Delta^1))\ar[d]_{\partial_N + \partial_{(\beta \t  \sigma)\circ \rho }+\partial_\DD}\ar[ddr]^{(\beta \t \sigma)\circ \rho + proj \circ \partial_N} & \\
\DD(A)\ar[r]^{\phi_{f}}\ar[drr]_{f^{01} \t \underline{01}^\#}  & 
         \DD(B\t N^*(\Delta^1))\ar[dr]|{proj} & \\
& & B\t \underline{01}^\# .
}$$

The triangle on the right obviously commutes. The commutativity of the triangle on the left defines $f^{01}$, the restriction of $f$ at the target in the component of $\underline{01}^\#$. The commutativity of the exterior diagram is equivalent to the commutativity of the inner square.

The commutativity of this diagram is equivalent to the equation:
\begin{equation}\label{eq2}
(f^{01}\t \underline{01}^\#) \circ (\partial_\DD + \partial_{\alpha}) = (\beta \t \sigma) \circ \rho \circ \phi_{f} +  (f^1-f^0)\t \underline{01}^\#.
\end{equation}

We want to construct the map $f^{01}$ by induction on the degree. We notice that in degree zero, $\DD_{[0]}(A)$ is reduced to $A$ and that $f^1_{[0]}-f^0_{[0]}=\psi -\psi=0$. Thus we define $f^{01}_{[0]} = 0$.

We now suppose by induction that $f^{01}$ is defined for degrees smaller than $d$ and we consider an element $\gamma (a_1, \ldots, a_n)$ where $\gamma$ lies in $\DD_{[d+1]}$. For this element, Equation \eqref{eq2} is equivalent to
$$(f^{01} \t \underline{01}^\# ) \left(\vcenter{
	\xymatrix@M=3pt@H=4pt@W=3pt@R=8pt@C=4pt{
	a_1 \ar@{-}[dr]\ar@{.}[rr] && a_n\ar@{-}[dl] \\
	&  \partial_{\DD}\gamma\ar@{-}[d] &  \\
	&&}  }  \right) 
+
\sum_{\nu_2} \sum_{k=1}^d (f^{01} \t \underline{01}^\# ) \left(\vcenter{
\xymatrix@M=3pt@H=4pt@W=3pt@R=8pt@C=4pt{
 	&a_*\ar@{.}[rr]\ar@{-}[dr] & \save+<-33pt,-11pt>*{\alpha \Biggl[} +<61pt,0pt>*{\Biggl]}\restore & a_*\ar@{-}[dl] \\
a_*\ar@{-}[drr]\ar@{.}[r]& & \gamma''_{[k]}\ar@{-}[d]  &\ar@{.}[r]& a_*\ar@{-}[dll] \\
	&& \gamma'\ar@{-}[d] && \\
	&&&}}\right)
$$ $$ =
\sum_{\stackrel{\nu}{\epsilon_* \in \{0,1,01\}}} \sum_{k=1}^{d+1} (\beta\t \sigma) \circ \rho \left(\vcenter{
\xymatrix@M=3pt@H=4pt@W=3pt@R=8pt@C=4pt{
 	\ar@{-}[dr] & \save+<-23pt,-11pt>*{f^{\epsilon_1} \Biggl[} +<49pt,0pt>*{\Biggl] \t \underline\epsilon^\#_1}\restore a_*\ar@{.}[r]\ar@{.}[l]& \ar@{-}[dl] &&&& \ar@{-}[dr] & \save+<-23pt,-11pt>*{f^{\epsilon_r} \Biggl[} +<49pt,0pt>*{\Biggl] \t \underline\epsilon^\#_r}\restore a_*\ar@{.}[r]\ar@{.}[l]& \ar@{-}[dl]\\
 & \gamma''_*\ar@{-}[drrr]  &\ar@{.}[rrrr]&&&&& \gamma''_*\ar@{-}[dlll] & \\
&	&&& \gamma'_{[k]}\ar@{-}[d] && \\
&	&&&&}}\right)
$$ $$ +
((f^1-f^0) \t \underline{01}^\# ) \left(\vcenter{
	\xymatrix@M=3pt@H=4pt@W=3pt@R=8pt@C=4pt{
	a_1 \ar@{-}[dr]\ar@{.}[rr] && a_n\ar@{-}[dl] \\
	& \gamma\ar@{-}[d] &  \\
	&&}  }  \right) 
$$
where $\gamma'_{[k]}$ and $\gamma''_{[k]}$ denote elements in $\DD_{[k]}$. 

\vspace{0.5cm}
The main difficulty in this equation (and the main difference with the study in Section \ref{construct1}) comes from the term 
$$\sum_{\nu} \sum_{k=1}^{d+1} (\beta\t \sigma) \circ \rho \left(\vcenter{
\xymatrix@M=3pt@H=4pt@W=3pt@R=8pt@C=4pt{
 	\ar@{-}[dr] & \save+<-23pt,-11pt>*{f^{\epsilon_1} \Biggl[} +<49pt,0pt>*{\Biggl] \t \underline\epsilon_1^\#}\restore a_*\ar@{.}[r]\ar@{.}[l]& \ar@{-}[dl] && \ar@{-}[dr] & \save+<-23pt,-11pt>*{f^{\epsilon_r} \Biggl[} +<49pt,0pt>*{\Biggl] \t \underline\epsilon_r^\#}\restore a_*\ar@{.}[r]\ar@{.}[l]& \ar@{-}[dl]\\
 & \gamma''_*\ar@{-}[drr]  &\ar@{.}[rr]&&& \gamma''_*\ar@{-}[dll] & \\
	&&& \gamma'_{[k]}\ar@{-}[d] && \\
	&&&&}}\right).$$

If $\gamma'$ is an element of $\DD_{[k]}, k \geq 2$, then the maps $f^{01}$ appearing in this term are applied to elements $\gamma''_{[\ell]}$ with $\ell \leq d-k$. Thus these terms are already known, according to the induction hypothesis.

If $\gamma'=p\t\pi$ is an element of $\DD_{[1]}$, we first  notice that $\rho (p \t \pi) = (p \t \pi) \t \pi$ for $p\t\pi \in \PP \boxtimes \EE \subset \tilde\PP$.
Then we can rewrite the term for $k=1$ as 
$$\beta \left(\vcenter{
\xymatrix@M=3pt@H=4pt@W=3pt@R=8pt@C=4pt{
 	\ar@{-}[dr] & \save+<-23pt,-11pt>*{f^{\epsilon_1} \Biggl[} +<40pt,0pt>*{\Biggl]}\restore a_*\ar@{.}[r]\ar@{.}[l]& \ar@{-}[dl] && \ar@{-}[dr] & \save+<-23pt,-11pt>*{f^{\epsilon_r} \Biggl[} +<40pt,0pt>*{\Biggl]}\restore a_*\ar@{.}[r]\ar@{.}[l]& \ar@{-}[dl]\\
 & \gamma''_*\ar@{-}[drr]  &\ar@{.}[rr]&&& \gamma''_*\ar@{-}[dll] & \\
	&&& p \t \pi\ar@{-}[d] && \\
	&&&&}}\right)
\t \sigma (\pi, \underline\epsilon_1^\#, \ldots, \underline\epsilon_r^\#)
$$
with $p$ in $\PP$ and $\pi$ in $\EE_0$. Exactly one of the $\underline\epsilon^\#$ has to be $\underline{01}^\#$ so that this term ends up in $B \t \underline{01}^\#$ (cf. the description of the action of $\EE_0$ on $N^*(\Delta ^1)$ in Section \ref{actionE}). Thus there is only one map $f^{01}$ involved. If this map $f^{01}$ is applied to an element  $\gamma''_{[\ell]}$ with $\ell \leq d-1$, the term is known. If this map $f^{01}$ is applied to an element $\gamma''_{[\ell]}$ with $\ell = d$, we know that all other $\gamma''$ must be in degree $0$, and thus the $f^\epsilon$ applied to these $\gamma''$ are just $\psi$.

\vspace{0.5cm}
Thus we rewrite Equation \eqref{eq2} as 
$$(f^{01} \t \underline{01}^\# ) \left(\vcenter{
	\xymatrix@M=3pt@H=4pt@W=3pt@R=8pt@C=4pt{
	a_1 \ar@{-}[dr]\ar@{.}[rr] && a_n\ar@{-}[dl] \\
	&  \partial_{\DD}\gamma\ar@{-}[d] &  \\
	&&}  }  \right) 
+
\sum_{\nu_2}  (f^{01}_{[d]} \t \underline{01}^\# ) \left(\vcenter{
\xymatrix@M=3pt@H=4pt@W=3pt@R=8pt@C=4pt{
 	&a_*\ar@{.}[rr]\ar@{-}[dr] & \save+<-33pt,-11pt>*{\alpha \Biggl[} +<61pt,0pt>*{\Biggl]}\restore & a_*\ar@{-}[dl] \\
a_*\ar@{-}[drr]\ar@{.}[r]& & \gamma''_{[1]}\ar@{-}[d]  &\ar@{.}[r]& a_*\ar@{-}[dll] \\
	&& \gamma'\ar@{-}[d] && \\
	&&&}}\right)
$$ $$
-
\sum_{\stackrel{\nu}{\epsilon_* \in \{0,1\}}}  (\beta \t \sigma) \circ \rho \left(\vcenter{
\xymatrix@M=3pt@H=4pt@W=3pt@R=8pt@C=4pt{
 	\ar@{-}[dr] & \save+<-21pt,-11pt>*{\psi \Biggl[} +<47pt,0pt>*{\Biggl] \t \underline\epsilon_1^\#}\restore a_*\ar@{.}[r]\ar@{.}[l]& \ar@{-}[dl] 
&&&& \ar@{-}[dr] & \save+<-25pt,-11pt>*{f^{01}_{[d]} \Biggl[} +<55pt,0pt>*{\Biggl] \t \underline{01}^\#}\restore a_*\ar@{.}[r]\ar@{.}[l]& \ar@{-}[dl]
&&&& \ar@{-}[dr] & \save+<-21pt,-11pt>*{\psi \Biggl[} +<47pt,0pt>*{\Biggl] \t \underline\epsilon_r^\#}\restore a_*\ar@{.}[r]\ar@{.}[l]& \ar@{-}[dl]\\
 & \gamma''_*\ar@{-}[drrrrrr]  &\ar@{.}[rrrr]&&&&& \gamma''_*\ar@{-}[d] &\ar@{.}[rrrr]&&&&& \gamma''_*\ar@{-}[dllllll] & \\
	&&&&&&& \gamma'_{[1]}\ar@{-}[d] && \\
&&&&	&&&&}}\right)
$$ $$ =
\sum_{\stackrel{\nu}{\epsilon_* \in \{0,1\}}} \sum_{\ell=0}^{d-1} (\beta \t \sigma) \circ \rho \left(\vcenter{
\xymatrix@M=3pt@H=4pt@W=3pt@R=8pt@C=4pt{
 	\ar@{-}[dr] & \save+<-23pt,-11pt>*{f^{\epsilon_1} \Biggl[} +<49pt,0pt>*{\Biggl] \t \underline\epsilon^\#_1}\restore a_*\ar@{.}[r]\ar@{.}[l]& \ar@{-}[dl] 
&&&& \ar@{-}[dr] & \save+<-25pt,-11pt>*{f^{01}_{[\ell]} \Biggl[} +<55pt,0pt>*{\Biggl] \t \underline{01}^\#}\restore a_*\ar@{.}[r]\ar@{.}[l]& \ar@{-}[dl]
&&&& \ar@{-}[dr] & \save+<-23pt,-11pt>*{f^{\epsilon_r} \Biggl[} +<49pt,0pt>*{\Biggl] \t \underline\epsilon_r^\#}\restore a_*\ar@{.}[r]\ar@{.}[l]& \ar@{-}[dl]\\
 & \gamma''_*\ar@{-}[drrrrrr]  &\ar@{.}[rrrr]&&&&& \gamma''_*\ar@{-}[d] &\ar@{.}[rrrr]&&&&& \gamma''_*\ar@{-}[dllllll] & \\
	&&&&&&& \gamma'_{[1]}\ar@{-}[d] && \\
&&&&	&&&&}}\right)
$$ $$+
\sum_{\stackrel{\nu}{\epsilon_* \in \{0,1,01\}}} \sum_{k=2}^{d+1} (\beta\t \sigma) \circ \rho \left(\vcenter{
\xymatrix@M=3pt@H=4pt@W=3pt@R=8pt@C=4pt{
 	\ar@{-}[dr] & \save+<-23pt,-11pt>*{f^{\epsilon_1} \Biggl[} +<49pt,0pt>*{\Biggl] \t \underline\epsilon^\#_1}\restore a_*\ar@{.}[r]\ar@{.}[l]& \ar@{-}[dl] &&&& \ar@{-}[dr] & \save+<-23pt,-11pt>*{f^{\epsilon_r} \Biggl[} +<49pt,0pt>*{\Biggl] \t \underline\epsilon^\#_r}\restore a_*\ar@{.}[r]\ar@{.}[l]& \ar@{-}[dl]\\
 & \gamma''_*\ar@{-}[drrr]  &\ar@{.}[rrrr]&&&&& \gamma''_*\ar@{-}[dlll] & \\
	&&&& \gamma'_{[k]}\ar@{-}[d] && \\
	&&&&&}}\right)
$$ $$-
\sum_{\nu_2} \sum_{k=2}^d  (f^{01}_{[d+1-k]} \t \underline{01}^\# ) \left(\vcenter{
\xymatrix@M=3pt@H=4pt@W=3pt@R=8pt@C=4pt{
 	&a_*\ar@{.}[rr]\ar@{-}[dr] & \save+<-33pt,-11pt>*{\alpha \Biggl[} +<61pt,0pt>*{\Biggl]}\restore & a_*\ar@{-}[dl] \\
a_*\ar@{-}[drr]\ar@{.}[r]& & \gamma''_{[k]}\ar@{-}[d]  &\ar@{.}[r]& a_*\ar@{-}[dll] \\
	&& \gamma'\ar@{-}[d] && \\
	&&&}}\right)
+
((f^1-f^0) \t \underline{01}^\# ) \left(\vcenter{
	\xymatrix@M=3pt@H=4pt@W=3pt@R=8pt@C=4pt{
	a_1 \ar@{-}[dr]\ar@{.}[rr] && a_n\ar@{-}[dl] \\
	& \gamma\ar@{-}[d] &  \\
	&&}  }  \right) .
$$

\vspace{0.5cm}
The last sum of the left hand side can be simplified. 
Actually, for a given $\gamma'_{[1]}=p \t \pi$, we have 

$$ 
(\beta \t \sigma) \circ \rho \left(\vcenter{
\xymatrix@M=3pt@H=4pt@W=3pt@R=8pt@C=4pt{
 	\ar@{-}[dr] & \save+<-21pt,-11pt>*{\psi \Biggl[} +<47pt,0pt>*{\Biggl] \t \underline\epsilon_1^\#}\restore a_*\ar@{.}[r]\ar@{.}[l]& \ar@{-}[dl] 
&&&& \ar@{-}[dr] & \save+<-25pt,-11pt>*{f^{01}_{[d]} \Biggl[} +<55pt,0pt>*{\Biggl] \t \underline{01}^\#}\restore a_*\ar@{.}[r]\ar@{.}[l]& \ar@{-}[dl]
&&&& \ar@{-}[dr] & \save+<-21pt,-11pt>*{\psi \Biggl[} +<47pt,0pt>*{\Biggl] \t \underline\epsilon_r^\#}\restore a_*\ar@{.}[r]\ar@{.}[l]& \ar@{-}[dl]\\
 & \gamma''_*\ar@{-}[drrrrrr]  &\ar@{.}[rrrr]&&&&& \gamma''_*\ar@{-}[d] &\ar@{.}[rrrr]&&&&& \gamma''_*\ar@{-}[dllllll] & \\
	&&&&&&& \gamma'_{[1]}\ar@{-}[d] && \\
&&&&	&&&&}}\right)
$$ $$
= \beta  \left(\vcenter{
\xymatrix@M=3pt@H=4pt@W=3pt@R=8pt@C=4pt{
 	\ar@{-}[dr] & \save+<-21pt,-11pt>*{\psi \Biggl[} +<38pt,0pt>*{\Biggl]}\restore a_*\ar@{.}[r]\ar@{.}[l]& \ar@{-}[dl] 
&&& \ar@{-}[dr] & \save+<-29pt,-11pt>*{f^{01}_{[d]} \Biggl[} +<51pt,0pt>*{\Biggl] }\restore a_*\ar@{.}[r]\ar@{.}[l]& \ar@{-}[dl]
&&& \ar@{-}[dr] & \save+<-21pt,-11pt>*{\psi \Biggl[} +<38pt,0pt>*{\Biggl]}\restore a_*\ar@{.}[r]\ar@{.}[l]& \ar@{-}[dl]\\
 & \gamma''_*\ar@{-}[drrrrr]  &\ar@{.}[rrr]&&&& \gamma''_*\ar@{-}[d] &\ar@{.}[rrr]&&&& \gamma''_*\ar@{-}[dlllll] & \\
	&&&&&& p \t \pi\ar@{-}[d] && \\
&&&&	&&&&}}\right)
\t \sigma (\pi,\epsilon_1^\#, \ldots, \underline{01}^\#, \ldots, \epsilon_r^\#)
$$ $$
= \beta_1  \left(\vcenter{
\xymatrix@M=3pt@H=4pt@W=3pt@R=8pt@C=4pt{
 	\ar@{-}[dr] & \save+<-21pt,-11pt>*{\psi \Biggl[} +<38pt,0pt>*{\Biggl]}\restore a_*\ar@{.}[r]\ar@{.}[l]& \ar@{-}[dl] 
&&& \ar@{-}[dr] & \save+<-29pt,-11pt>*{f^{01}_{[d]} \Biggl[} +<51pt,0pt>*{\Biggl] }\restore a_*\ar@{.}[r]\ar@{.}[l]& \ar@{-}[dl]
&&& \ar@{-}[dr] & \save+<-21pt,-11pt>*{\psi \Biggl[} +<38pt,0pt>*{\Biggl]}\restore a_*\ar@{.}[r]\ar@{.}[l]& \ar@{-}[dl]\\
 & \gamma''_*\ar@{-}[drrrrr]  &\ar@{.}[rrr]&&&& \gamma''_*\ar@{-}[d] &\ar@{.}[rrr]&&&& \gamma''_*\ar@{-}[dlllll] & \\
	&&&&&& p\ar@{-}[d] && \\
&&&&	&&&&}}\right)
\t \sigma (\pi,\epsilon_1^\#, \ldots, \underline{01}^\#, \ldots, \epsilon_r^\#).
$$ 

Only one choice of $\epsilon$'s will give a non-zero term: the one where after composition with the permutation $\pi$, the sequence is $(\underline 0^\#,\ldots, \underline 0^\#, \underline {01}^\#, \underline 1^\#, \ldots, \underline 1^\#)$, according to the action of $\EE_0$ on $N^*(\Delta^1)$.

\vspace{0.5cm}
Thus we finally get 
$$(f^{01} \t \underline{01}^\# ) \left(\vcenter{
	\xymatrix@M=3pt@H=4pt@W=3pt@R=8pt@C=4pt{
	a_1 \ar@{-}[dr]\ar@{.}[rr] && a_n\ar@{-}[dl] \\
	&  \partial_{\DD}\gamma\ar@{-}[d] &  \\
	&&}  }  \right) 
+
\sum_{\nu_2}  (f^{01}_{[d]} \t \underline{01}^\# ) \left(\vcenter{
\xymatrix@M=3pt@H=4pt@W=3pt@R=8pt@C=4pt{
 	&a_*\ar@{.}[rr]\ar@{-}[dr] & \save+<-33pt,-11pt>*{\alpha_1 \Biggl[} +<61pt,0pt>*{\Biggl]}\restore & a_*\ar@{-}[dl] \\
a_*\ar@{-}[drr]\ar@{.}[r]& & \gamma''_{[0]}\ar@{-}[d]  &\ar@{.}[r]& a_*\ar@{-}[dll] \\
	&& \gamma'\ar@{-}[d] && \\
	&&&}}\right)
$$ $$
-
\sum_{\nu} \beta_1  \left(\vcenter{
\xymatrix@M=3pt@H=4pt@W=3pt@R=8pt@C=4pt{
 	\ar@{-}[dr] & \save+<-21pt,-11pt>*{\psi \Biggl[} +<38pt,0pt>*{\Biggl]}\restore a_*\ar@{.}[r]\ar@{.}[l]& \ar@{-}[dl] 
&&& \ar@{-}[dr] & \save+<-29pt,-11pt>*{f^{01}_{[d]} \Biggl[} +<51pt,0pt>*{\Biggl] }\restore a_*\ar@{.}[r]\ar@{.}[l]& \ar@{-}[dl]
&&& \ar@{-}[dr] & \save+<-21pt,-11pt>*{\psi \Biggl[} +<38pt,0pt>*{\Biggl]}\restore a_*\ar@{.}[r]\ar@{.}[l]& \ar@{-}[dl]\\
 & \gamma''_*\ar@{-}[drrrrr]  &\ar@{.}[rrr]&&&& \gamma''_*\ar@{-}[d] &\ar@{.}[rrr]&&&& \gamma''_*\ar@{-}[dlllll] & \\
	&&&&&& \gamma'_{[0] |\PP}\ar@{-}[d] && \\
&&&&	&&&&}}\right)
\t  \underline{01}^\#
$$ $$ =
\sum_{\stackrel{\nu}{\epsilon_3 \in \{0,1\}}} \sum_{\ell=0}^{d-1} (\beta_0 \t \sigma) \circ \rho \left(\vcenter{
\xymatrix@M=3pt@H=4pt@W=3pt@R=8pt@C=4pt{
 	\ar@{-}[dr] & \save+<-23pt,-11pt>*{f^{\epsilon_1} \Biggl[} +<49pt,0pt>*{\Biggl] \t \underline\epsilon_1^\#}\restore a_*\ar@{.}[r]\ar@{.}[l]& \ar@{-}[dl] 
&&&& \ar@{-}[dr] & \save+<-25pt,-11pt>*{f^{01}_{[\ell]} \Biggl[} +<55pt,0pt>*{\Biggl] \t \underline{01}^\#}\restore a_*\ar@{.}[r]\ar@{.}[l]& \ar@{-}[dl]
&&&& \ar@{-}[dr] & \save+<-23pt,-11pt>*{f^{\epsilon_r} \Biggl[} +<49pt,0pt>*{\Biggl] \t \underline\epsilon_r^\#}\restore a_*\ar@{.}[r]\ar@{.}[l]& \ar@{-}[dl]\\
 & \gamma''_*\ar@{-}[drrrrrr]  &\ar@{.}[rrrr]&&&&& \gamma''_*\ar@{-}[d] &\ar@{.}[rrrr]&&&&& \gamma''_*\ar@{-}[dllllll] & \\
	&&&&&&& \gamma'_{[1]}\ar@{-}[d] && \\
&&&&	&&&&}}\right)
$$ $$+
\sum_{\stackrel{\nu}{\epsilon_* \in \{0,1,01\}}} \sum_{k=1}^{d+1} (\beta\t \sigma) \circ \rho \left(\vcenter{
\xymatrix@M=3pt@H=4pt@W=3pt@R=8pt@C=4pt{
 	\ar@{-}[dr] & \save+<-23pt,-11pt>*{f^{\epsilon_1} \Biggl[} +<49pt,0pt>*{\Biggl] \t \underline\epsilon_1^\#}\restore a_*\ar@{.}[r]\ar@{.}[l]& \ar@{-}[dl] &&&& \ar@{-}[dr] & \save+<-23pt,-11pt>*{f^{\epsilon_r} \Biggl[} +<49pt,0pt>*{\Biggl] \t \underline\epsilon_r^\#}\restore a_*\ar@{.}[r]\ar@{.}[l]& \ar@{-}[dl]\\
 & \gamma''_*\ar@{-}[drrr]  &\ar@{.}[rrrr]&&&&& \gamma''_*\ar@{-}[dlll] & \\
	&&&& \gamma'_{[k]}\ar@{-}[d] && \\
	&&&&}}\right)
$$ $$-
\sum_{\nu_2} \sum_{k=2}^d  (f^{01}_{[d+1-k]} \t \underline{01}^\# ) \left(\vcenter{
\xymatrix@M=3pt@H=4pt@W=3pt@R=8pt@C=4pt{
 	&a_*\ar@{.}[rr]\ar@{-}[dr] & \save+<-33pt,-11pt>*{\alpha \Biggl[} +<61pt,0pt>*{\Biggl]}\restore & a_*\ar@{-}[dl] \\
a_*\ar@{-}[drr]\ar@{.}[r]& & \gamma''_{[k]}\ar@{-}[d]  &\ar@{.}[r]& a_*\ar@{-}[dll] \\
	&& \gamma'\ar@{-}[d] && \\
	&&&}}\right)
+
((f^1-f^0) \t \underline{01}^\# ) \left(\vcenter{
	\xymatrix@M=3pt@H=4pt@W=3pt@R=8pt@C=4pt{
	a_1 \ar@{-}[dr]\ar@{.}[rr] && a_n\ar@{-}[dl] \\
	& \gamma\ar@{-}[d] &  \\
	&&}  }  \right) 
$$
where $\gamma'_{[1] |\PP}$ denotes the component in $\PP$ of $\gamma'_{[1]} \in \PP\boxtimes\EE$.

All the terms in the right hand side are already known. 
The left hand side can be identified with $\partial(f^{01}_{[d]}(\gamma)\t \underline{01}^\#) $ where $\partial$ is the differential in $\Der_{\tilde\PP}(\tilde\PP(\DD (A), \partial_{\alpha_1}), (B \t \underline{01}^\#,\beta_1))$.
Note that these derivations take into account only the restricted structures $\alpha_1$ and $\beta_1$, and not the full structures $\alpha$ and $\beta$.

\vspace{0.5cm}
We have proved
\begin{theo}
If the cohomology group $H^1\Der_{\tilde\PP}(\tilde\PP(\DD (A), \partial_{\alpha_1}), (B \t \underline{01}^\#,\beta_1))$ is equal to $0$, we can construct a map $f^{01}_{[d]}$ (i.e. continue our induction), and hence a map $\phi_f$ answering the initial problem. 
\end{theo}

\vspace{0.5cm}

We now relate this cohomology group with one group of $\Gamma$-cohomology:

\begin{coro}\label{thunicite}
The obstruction to the existence of a homotopy of two realizations of a morphism lies in $H\Gamma_\PP^0(H_*A,H_*B)$. 
\end{coro}

\begin{proof}
The proof is almost the same as the proof of Theorem \ref{threal}. The only difference is that working with $B \t \underline{01}^\#$ instead of $B$ creates a shift in the degree of the cohomology group.
\end{proof}

Remarks similar to the ones in Section \ref{remarks} for the realization of morphisms can also be stated for the uniqueness of the realizations.

\vspace{1cm}

\section*{Acknowledgements}
I would like to thank David Chataur and Benoit Fresse for many useful discussions on the matter of this article. I am also grateful to Muriel Livernet and Birgit Richter for their careful reading and their remarks.


\begin{thebibliography}{}
    \bibitem[BF]{3BF} C. Berger, B. Fresse, \textit{Combinatorial operad actions on cochains}, Math.Proc. Camb. Phil. Soc \textbf{137} (2004), 135-174.
    \bibitem[BM1]{3BM1} C. Berger, I. Moerdijk, \textit{Axiomatic homotopy theory for operads}, Comment. Math. Helv. \textbf{78} (2003), 805-831.
    \bibitem[BM2]{3BM2} C. Berger, I. Moerdijk, \textit{The Boardman-Vogt resolution of operads in monoidal model categories}, Topology \textbf{45} (2006), 807-849.
    \bibitem[BDG]{3BDG} D. Blanc, W. Dwyer, P. Goerss, \textit{The realization space of a $\Pi$-algebra: a moduli problem in algebraic topology}, Topology \textbf{43} (2004), 857-892.
   \bibitem[DS]{3DS} W. Dwyer, J. Spalinski, \textit{Homotopy theories and model categories, in} Handbook of Algebraic Topology, Elsevier, 1995, 73-126.
   \bibitem[F1]{3BouquinBenoit} B. Fresse, Modules over operads and functors, Lecture Notes in Mathematics \textbf{1967}, Springer Verlag, 2009.
   \bibitem[F2]{3Arolla2008} B. Fresse, Operadic cobar constructions, cylinder objects and homotopy morphisms of algebras over operads, in "Alpine perspectives on algebraic topology (Arolla, 2008)", Contemp. Math. \textbf{504}, Amer. Math. Soc. (2009), 125-189.
   \bibitem[F3]{3Evanston} B. Fresse, {\it Koszul duality of operads and homology of partition posets, in} "Homotopy theory and its applications (Evanston, 2002)", Contemp. Math. \textbf{346} (2004), 115-215. 
   \bibitem[F4]{3Fgeorg} B. Fresse, \textit{Props in model categories and homotopy invariance of structures}, Georgian Math. J. \textbf{17} (2010), 79-160.
   \bibitem[GJ]{3GJ} E. Getzler, J. D. S. Jones, \textit{Operads, homotopy algebra and iterated integrals for double
loop spaces}, \texttt{hep-th/9403055} (1994).  
   \bibitem[GK]{3GK} V. Ginzburg, M. Kapranov, {\it Koszul duality for operads}, Duke Math. J.  \textbf{76} (1995), 203-272.
   \bibitem[GH]{3GH} P. Goerss, M. Hopkins,  Moduli spaces of commutative ring spectra, in ``Structured ring spectra'' London Math. Soc. Lecture Note Ser. \textbf{315} (2004), 151-200. 
   \bibitem[Hin]{3Hinich} V. Hinich, \textit{Homological algebra of homotopy algebras},  Comm. Algebra  \textbf{25}  (1997),  no. 10, 3291-3323.
   \bibitem[Hir]{3Hirsch} P. Hirschhorn, Model categories and their localizations, Mathematical Surveys and Monographs, \textbf{99}, 2003.
   \bibitem[Hof]{3Hof} E. Hoffbeck, \textit{Gamma-homology of algebras over an operad}, Algebraic \& Geometric Topology \textbf{10} (2010), 1781-1806.
   \bibitem[Hov]{3Hovey} M. Hovey, Model categories, Mathematical Surveys and Monographs, \textbf{63}, 1999.
   \bibitem[HS]{3HS} S. Halperin, J. Stasheff, \textit{Obstructions to homotopy equivalences}, Adv. in Math. \textbf{32} (1979), 233-279. 
   \bibitem[Kad]{3Kad} T. Kadeishvili, On the homology theory of fibre spaces, in ``International Topology
Conference (Moscow State Univ., Moscow, 1979)'', Uspekhi Mat. Nauk \textbf{35} (1980), 183-188.
   \bibitem[Liv]{3Liv} M. Livernet, \textit{On a plus-construction for algebras over an operad}, K-theory  \textbf{18} (1999), 317-337.
   \bibitem[MS]{3MS} M. Markl, S. Shnider \textit{Associahedra, cellular $W$-construction and products of $A_\infty$-algebras}, Trans. Amer. Math. Soc.  \textbf{358}  (2006), 2353--2372 (electronic).
   \bibitem[Rob]{3Rob} A. Robinson, \textit{Gamma homology, Lie representations and $E\sb \infty$ multiplications},  Invent. Math.  \textbf{152}  (2003), 331-348. 
   \bibitem[RW]{3RW} A. Robinson, S. Whitehouse, \textit{Operads and $\Gamma$-homology of commutative rings},  Math. Proc. Cambridge Philos. Soc.  \textbf{132 } (2002), 197-234. 
\end{thebibliography}
\end{document}